\documentclass[a4paper,reqno]{amsart}

\usepackage{amsmath}
\usepackage{paralist}
\usepackage{graphics}
\usepackage{epsfig}
\usepackage{amsfonts}
\usepackage{amssymb}
\usepackage{hyperref,cite}
\usepackage{xcolor}

\usepackage{amsthm}
\usepackage{enumerate}
\usepackage[mathscr]{eucal}
\usepackage{eqlist}
\usepackage[misc]{ifsym}
\usepackage{yhmath}

\def\R{{\Bbb R}}

\def\ifl{\iffalse }

\def\bb{\mathbb}
\def\I{\vert}

\newtheorem{prop}{Proposition}[section]
\newtheorem{theo}{Theorem}[section]
\newtheorem{lem}{Lemma}[section]

\begin{document}

\baselineskip=1.2\baselineskip
\title[  Denjoy Domains and BMOA]
{ Denjoy Domains and BMOA}

\author {Shengjin Huo}
\address{School of Mathematical Sciences, Tiangong University, Tianjin 300387, China} \email{huoshengjin@tiangong.edu.cn}
\author{Michel Zinsmeister }
\address{Institut Denis Poisson, UFR Sciences et Techniques, Universit\'{e} d'Orl\'{e}ans
BP 6749 45067 Orl\'{e}ans Cedex 2, France, zins@univ-orleans.fr}



\thanks{ This work was supported by the National Natural Science Foundation of China (Grant No.11401432).}


\begin{abstract}
 A Denjoy domain is a plane domain whose complement is a closed subset $E$ of the extended real line $\bar{R}$ containing $\infty$ : such a domain is called Carleson-homogeneous if there exists $C>0$ such that for all $z\in E$ and $r>0$, one has $\vert E\cap [z-r,z+r]\vert\geq Cr$, where $\vert\cdot\vert$ is the Lebesgue measure on the line. We prove that if $U=\bar{ \mathbb C}\backslash K$ is a Carleson-homogeneous Denjoy domain then, if $f$ stands for one of its universal coverings, $\log {f'}\in BMOA.$ In order to prove this result, we develop ideas from
 [On Carleson measures induced by Beltrami coefficients being compatible with Fuchsian groups, Ann. Fenn. Math. 46(2021),67-77] leading to a general theorem about planar domains giving sufficient conditions ensuring that $\log {f'}\in BMOA$ for any universal covering $f.$
\end{abstract}

\maketitle

\section{Background and notations}
In this paper $\bb C$ will denote the complex plane and $\bar{\bb C}$ the Riemann sphere. If $z\in \bb C$ and $r>0, \,D(z,r)=\{\zeta\in \bb C:\,\I\zeta-z\I<r\}$. The unit disk, $D(0,1)$ will be denoted by $\bb D$ and $\mathbb{D}^{*}=\bb C\backslash\bar{\bb D}$. If $E\subset \bb C$, diam$(E)$ stands for the Euclidian diameter
$$\mathrm{diam}(E)=d_1(E)=\sup\limits_{z,\zeta\in E}|z-\zeta|.$$
We generalize this definition to $n>1$ by
$$d_n(E)=\sup\limits_{z_1,..,z_n\in E}\left(\prod\limits_{j\neq k}|z_j-z_k|\right)^\frac{2}{n(n-1)},$$
a quantity which happens to be nonincreasing in $n$. We define then
cap$(E)$ as $\lim\limits_{n\to\infty}d_n(E)$. It is the transfinite diameter of $E$ which happens to be equal to the logarithmic capacity that is usually defined by potential theory techniques that we will not need, see \cite{Ra}.

A Jordan curve $\Gamma: I\to \bb C$ in the plane is said to be a quasicircle if there exists $C>0$ such that given two points $z,\zeta\in \Gamma(I)$ then $$\min\{\mathrm{diam}(\gamma_1),\mathrm{diam}(\gamma_2)\}\leq C|z-\zeta|,$$ where $\gamma_1,\gamma_2$ are the two subarcs of $\Gamma$ defined by $z,\zeta.$

A rectifiable curve $\gamma:\,I\to \bb C$ is said to be Ahlfors-David regular if there exists $C>0$ such that for every $z\in \gamma(I)$ and every $0<r<$diam$(\gamma(I))$, $$\text{length}(D(z,r)\cap\gamma(I))\leq Cr,$$ where if $F\subset \gamma(I)$ is, say, a Borel set, $$\text{length}~(F)=\int_{\gamma^{-1}(F)}|\gamma'(s)|ds.$$

Finally a Jordan curve $\Gamma: I\to \bb C$ is said to be a chord-arc curve if it is at the same time Ahlfors-David regular and a quasidisk. Equivalently $\Gamma$ is chord-arc if given two points $z,\zeta\in \Gamma(I),$ then $$\min\{\mathrm{length}(\gamma_1),\mathrm{length}(\gamma_2)\}\leq C|z-\zeta|,$$ where $\gamma_1,\gamma_2$ are the two subarcs of $\Gamma$ defined by $z,\zeta$.

We end this introductory section by recalling some facts from Riemann surfaces. A Riemann surface $\mathcal{R}$ is a one-dimensional complex manifold. The universal cover $\mathcal{R}^*$ of $\mathcal{R}$ is conformally equivalent to the Riemann sphere, the complex plane or the unit disk, see \cite{Ah}. In the latter case we say that the surface is hyperbolic, and the surface is then conformally equivalent to $\bb D/G,$ where $G$ is a Fuchsian group, that is a discrete group of automorphisms of the disk. We moreover call universal covering any automorphism from $\bb D$ onto $\mathcal{R}^*$.

In this paper we will be mainly interested in planar Riemann surfaces, which are nothing else but the subdomains of the Riemann sphere. Such a surface is hyperbolic if and only if the complement of $\Omega$ in the sphere contains more than $ 3$ points. In the case of a planar Riemann surface, a universal covering (or more precisely its projection on the Riemann sphere) is then a holomorphic function $F$ from $\mathbb D$ onto $\Omega$, locally injective and such that $F\circ \gamma=F,\, \gamma\in G$, a Fuchsian group such that $\Omega$ is conformally equivalent to $\bb D/G$. Notice that in the case $\Omega$ is simply-connected this map is a Riemann map, that is a bi-holomorphism. The universal covering may thus be seen as a Riemann map for general planar domains.

\section{Main result and motivation}
 A Denjoy domain is a plane domain whose complement is a closed subset $E$ of the extended real line $\bar{R}$ containing $\infty$. These domains have been intensively studied, the reason being that they are in some sense the simplest infinitely connected domains in the plane. For example, Rubel and Ryff (\cite{RR})have shown how to construct quasi-explicitely the uniformizing Fuchsian group of such a domain. Also (see \cite{C}, \cite{GJ}, \cite{JM}) these domains were the first infinitely connected domains for which the corona property was proven to hold.

We recall that a plane domain $U$ has the corona property if for any $n>0$ and for any $n$-tuple $(f_1,..,f_n)$ of holomorphic functions in $U$ such that there exists $\delta>0$ with
$$ \delta\leq \inf\limits_{z\in U}(\max\limits_{j=1,..,n}\I f_j(z)\I)\leq\sup\limits_{z\in U}(\max\limits_{j=1,..,n}\I f_j(z)\I)\leq 1,$$
there exist $n$ bounded holomorphic functions $g_1,..g_n$ in $U$ such that $$\forall z\in U,\,f_1(z)g_1(z)+..+f_n(z)g_n(z)=1.$$ In \cite{GJ}, Garnett and Jones  eventually proved that all Denjoy domains possess the corona property, but Carleson $($\cite{C}$)$ has been the first to raise the question and to solve it in the special case of what is now called Carleson-homogeneous Denjoy domains, a notion that will happen to be central in the present paper.

A Denjoy domain $\bar{\bb C}\backslash E$ is said to be Carleson-homogeneous if, writing $|\cdot|$ for the Lebesque measure on $\mathbb{R}$, there is a constant $C$ such that if for all $x\in E$ and $t>0$,
$$|(x-t,x+t)\cap E|\geq Ct.$$
The idea in Carleson approach is that the thickness of the boundary implies that a Carleson-homogeneous Denjoy domain "behaves" like a simply-connected one, for which Carleson  \cite{Ca} had previously proven the corona property. Here the "thickness" of the boundary is measured by the Lebesgue measure. There are other ways to measure this thickness, as for example by logarithmic capacity: a plane domain $U$ is said to be uniformly perfect if there exists $C>0$ such that $\text{for all}\, z\in \partial U$ and $0<t<\mathrm{diam}(U),$
$$ \mathrm{cap}(D(z,t)\cap E)\geq Ct,$$
where $D(z,t)$ stands for the disk of center $z$ and radius $t$ and cap is logarithmic capacity. If $V\neq \bb C$ is any simply-connected plane domain a famous theorem of Koebe asserts that, $f:\bb D\to  V$ being a Riemann map, $\log f'\in \mathcal{B}$, the Bloch class defined as
$$\mathcal{B}=\{b\, \mathrm{holomorphic\,in}\,\bb D:\,\sup\limits_{z\in\bb D}(1-\I z\I)\I b'(z)\I<+\infty\}.$$

 Pommerenke (\cite{P2}) has proven that if $f$ is the universal covering of a hyperbolic plane domain $V,$ then $\log f'\in \mathcal{B}$ if and only if $V$ is uniformly perfect, so that we may say that uniformly perfect domains "behave" with this respect like simply-connected ones.

There is yet another instance of this comparison between simply and multiply-connected domains, namely the Hayman-Wu theorem:
\begin{theo} There exists a (universal) constant $C>0$ such that for any hyperbolic simply-connected planar domain $V$, for any holomorphic bijection $f$ from $\bb D$ onto $V$ and for any line $L$, length$(f^{-1}(L\cap V))\leq C$.
\end{theo}
By analogy we will say that a planar hyperbolic domain $V$ has Hayman-Wu property if there exists a constant $C>0$ such that, $f$ being any universal covering of $V$, length$(f^{-1}(L\cap V))\leq C$ for all lines $L$. Fernandez and Hamilton $($ \cite{FH}$)$ proved the following:
\begin{theo} Let $U$ be a Denjoy domain; then $U$ has Hayman-Wu property if and only if it is Carleson-homogeneous.
\end{theo}
Since if $f$ is a universal covering and $T$ is an automorphism of the unit disk then $f\circ T$ is another universal covering, we see that for a line $L$, not only $f^{-1}(L)$ has finite length but also there exists a constant $C>0$ such that length$(T(f^{-1}(L))\leq C$, meanning that arclength on $f^{-1}(L)$ is a Carleson measure (see below the definition and other instances of this property).

In the spirit of what preceeds, here is the main result of the paper:
\begin{theo} Let $U$ be a  Carleson-homogeneous Denjoy domain and $f$ one of its universal coverings, then $\log f'\in BMOA(\bb D)$.
\end{theo}
In this statement, $ BMOA(\bb D)$ is the subset of the Hardy space $H^2(\bb D)$ whose boundary values belong to the John-Nirenberg space $BMO$ (bounded mean oscillations). Notice that $BMOA(\bb D)$ is a subspace of $\mathcal{B}$ and that the theorem is obvious if $V$ is simply-connected (i.e. if $E$ is an interval).

\section{Fuchsian groups}
Let $G$ be a uniformizing group for a domain $\Omega$: it is a Fuchsian group, meanning that it acts on the unit disk $\mathbb{D}$ of the plane properly discontinuously and freely, and it uniformizes $\Omega$ in the sense that $\Omega$ is conformally equivalent to the Riemann surface $\bb D/G.$ For an element $g$ in $G$, we denote by $\mathcal{H}_{z}(g)$ the closed hyperbolic half-plane containing $z$, bounded by the perpendicular bisector of the hyperbolic segment $[z, g(z)]$. The Dirichlet fundamental domain $\mathcal{F}_{z}(G)$ of $G$ centered at $z$ is the intersection of all the sets $\mathcal{H}_{z}(g)$ with $g$ in $G\backslash\{Id\}$. For simplicity, in this paper we use the notation $\mathcal{F}$ for the Dirichlet fundamental domain $\mathcal{F}_{0}(G)$ of $G$ centered at $z=0.$  It is easy to see that a Dirichlet domain is a convex subset  of $\mathbb{D}$ for the hyperbolic metric. For the Dirichlet fundamental domain $\mathcal{F}$, let $\mathcal{F}^{\circ}$ denote its interior and $\bar{\mathcal{F}}$ its closure. Let $ g$ be a nontrivial element of $G$. When the intersection of $\mathcal{F} $ and $g(\mathcal{F}) $ is non-empty, it is contained in  the perpendicular bisector of the segment $[z, g(z)]_{h}$. This intersection is a point, a non-trivial geodesic segment, a geodesic ray or a geodesic. In the latter three cases, we say this intersection is an edge. The vertices are the endpoints of the edges. An infinity vertex is a vertex  contained in the unit circle $\partial \mathbb{D}.$ When $G$ is of first kind (i.e. if its limit set is the whole unit circle), the  infinity vertex set $\mathcal{F}(\infty) $ is at most countable ( possibly empty), see \cite{B}. 

Let $\lambda(E)$ denote the Hausdorff linear measure of a set $E.$
A Fuchsian group $G$ is said to be of weak-finite length type if  $$\sum_{g\in G} \lambda(g(\partial \mathcal{F}))<+\infty.$$
We also say that a  hyperbolic planar domain $\Omega$  is of finite length type if the universal cover group is of finite length type. Notice that in the case of a Denjoy domain, the quantity $$\sum_{g\in G} \lambda(g(\partial \mathcal{F}))$$ is precisely the length of $f^{-1}(\R)$ for some universal covering $f$. We will say that $G$ is of strong-finite length type if there exists $C>0$ such that the length of $f^{-1}(\R)$ is bounded by $C$ for all universal covering of $\Omega.$   As we will see later this is equivalent to saying that arclength on $\cup_G g(\partial \mathcal{F})$ is a Carleson measure.

It is interesting to connect this notion of finite length type with the exponent of convergence of the Fuchsian group, defined as the infinimum of the $\alpha's$ such that $$\sum\limits_{g\in G}(1-\I g(0)\I)^\alpha<\infty.$$
We denote this number by $\delta(G)$ or $\delta(\Omega)$. Fernandez and Hamilton \cite{FH} have shown that if $\delta(G)<1/2$ then $G$ is of finite length type. Conversely Fernandez has shown that if $\Omega$ is uniformly perfect then $\delta(\Omega)<1$. In particular $G$ is of convergence type, that is $\sum_G(1-\I g(0)\I)<\infty$. In general if $G$ is not co-compact, it is easy to see that $$(1-\I g(0)\I)\leq C  \,\text{length}\,(g(\partial \mathcal{F})),$$ which gives another proof of the preceeding assertion. If $G$ is co-compact then the  two quantities are of the same order, but in this case the group is of divergence type. We do not know if there are convergence-type groups for which these quantities are comparable. Also convergence type does not imply finite length type. An example is given by $\Omega=\bar{\bb C}\backslash K$ where $K$ is the triadic Cantor set : $\Omega$ is then uniformly perfect and thus of convergence type but is not of finite length type, not being Carleson-homogeneous.

\section{Carleson measures}
Recall that a positive measure $\lambda$ defined in a domain $\Omega$ is called a Carleson measure if
$$\parallel\lambda\parallel^{2}=\sup\{\frac{\lambda(\Omega\cap D(z,r)):z\in\partial\Omega}{r}, 0<r<\mathrm{diam}(\partial\Omega)\}<+\infty.$$
In order to motivate what follows, let us start with an example coming from Teichm{\"u}ller theory. If $G$ is a Fuchsian group and $\mu(z)$ a bounded measurable function on $\Delta$ which satisfies $$||\mu(z)||_{\infty}<1 ~~\text {and} ~~\mu(z)=\mu(g(z))\overline{g'(z)}/g'(z)$$ for every $g\in G$, then we say that $\mu$ is a  $G$-compatible Beltrami coefficient (or complex dilatation). We denote by $M(G)$ the set of all $G$-compatible Beltrami coefficients.

We recall that for a Fuchsian group $G$, the unit disk $\mathbb{D}$ can be tessellated  by the images any Dirichlet fundamental domain $\mathcal{F}(G)$ of $G$. Is it possible  to check  that the measure $$\lambda_{\mu}=|\mu|^{2}/(1-|z|^{2})dxdy$$ induced by $\mu\in M(G)$ is in $CM(\mathbb{D})$  directly from its value on the Dirichlet fundamental domain $\mathcal{F}_{z}(G)$ of $G$?

Very recently, the first author in \cite{H} has given a positive answer for the finitely generated groups of the second kind (i.e. not of the first kind) without parabolic elements.

\begin{theo}\cite{H}\label{main1}
Let $G$ be finitely generated Fuchsian of the second kind without parabolic elements  and $\mathcal{F}$ be the Dirichlet  fundamental domain of $G$ centered at $0$. Let $\mu$ be a $G$-compatible Beltrami coefficient in $\in M(G)$. Then the measure $|\mu|^{2}(1-|z|^{2})^{-1}$
is a Carleson measure on the unit disk if and only if the restriction of the measure $|\mu|^{2}(1-|z|^{2})^{-1}$ to a  Dirichlet fundamental domain $\mathcal{F}(\infty)$ is a Carleson measure on $\mathcal{F}.$
\end{theo}

 In this theorem, the assumption that the finite generated Fuchsian group $G$ does not contain parabolic elements is essential. A counter- example is given by  a cyclic parabolic group, for instance the covering group of a punctured disk. In this case the fundamental domain in the upper half-plane is a vertical strip and setting $|\mu|$ to be constant in the part of the strip above height $1$  to provide a counterexample, since if $|\mu|$ is constant on a horodisk, then it does not give rise to a Carleson measure. Notice also that in this theorem the associated Riemann surface need not be planar.

Let us generalize slightly the problem. In what preceeds we considered measures of the form
$$ m=|\mu|^{2}(1-|z|^{2})^{-1} dxdy$$
where $\mu\in M(G)$. It is easy to see that in this case
$$m|\gamma(\mathcal{F})=\gamma^*(|\gamma'|(m|\mathcal{F})).$$
Let $\nu$ be a positive and finite measure on $\bar{\mathcal{F}}$; we define the measure $\tilde{\nu}$ on the closed unit disk as
$$\tilde{\nu}=\sum\limits_{\gamma\in G}\gamma^*(|\gamma'|\nu),$$
and we ask the question: for which groups is it true that $\tilde{\nu}$ is a Carleson measure on the closed disk if $\nu$ is a Carleson measure on $\mathcal{F}$?
It is thus natural to investigate the groups $G$ satisfying the following property, where $\mathcal{F}$ is the Dirichlet fundamental domain:
$$(H)\; \;\nu\in CM(\mathcal{F})\Rightarrow \tilde{\nu}\in CM(\mathcal{\bb D}).$$
\begin{prop} The properties $(H)$ and $(SFLT)$ (strongly-finite length type) are equivalent for a Fuchsian group $G$.
\end{prop}
Proof: In order to prove that $(H)\Rightarrow (SFLT)$ it suffices to apply $(H)$ to $$\nu=\text{Arclength}~~ (\partial \mathcal{F}).$$

Suppose now that $(SFLT)$ holds and let $\nu\in CM(\mathcal{F})$. Let $T$ be any automorphism of the disk : we may write
\begin{align*}
\int |T'|d\tilde{\nu}&=\sum\limits_{\gamma\in G}\int_{\gamma(\mathcal{F})}|T'|d\gamma^*(|\gamma'|\nu)\\
&=\sum\limits_{\gamma\in G}\int_{\mathcal{F}}|(T\circ\gamma)'|d\nu
\end{align*}
and this last quantity is bounded by a constant time the length of $T(\cup_G\gamma(\partial\mathcal{F}))$ which is finite (uniformly in $T$) by $(SFLT)$: this proves that $\tilde{\nu}\in CM(\bb D)$ and thus $(H)$.

We finish this section with an application of the last proposition: let $\Omega$ be a hyperbolic planar domain which is uniformly perfect, $G$ an uniformizing Fuchsian group and $f:\,\bb D\to \Omega$ a universal covering of $\Omega$. Gonzalez  has shown that under these hypothesis there exists a fundamental domain $\mathcal{F}$ for $G$ which is a chord-arc domain. Then $f$ is a conformal mapping onto a simply connected domain $U\subset\Omega$ which may be seen as a fundamental domain on the Riemann surface $\Omega$.
\begin{theo} If $G$ has property $(H)$ for $\mathcal{F}$ and if $U$ is Ahlfors-David regular then $\log f'\in BMOA(\bb D)$.
\end{theo}
\begin{proof} Let $\varphi$ be a Riemann map from the disk onto $\mathcal{F}$. By \cite{Z}, $\log(f\circ\varphi)'\in BMOA(\bb D)$ as well as $\log\varphi'$. It follows that $\log f'\circ\varphi\in BMOA(\bb D)$. By definition this means that $\log f'\in BMOA(\mathcal{F})$. By results in \cite{Z} it implies that $$d(z,\partial \mathcal{F})^3|S_f(z)|^2dxdy\in CM(\mathcal{F}).$$

At this point we use the already mentioned result of Pommerenke that $\log f'\in\mathcal{B}$: it implies that $(1-|z|^2)^3|S_f(z)|^2dxdy\in CM(\mathcal{F})$. Using now property $(H)$ an elementary computation shows that it implies that $$(1-|z|^2)^3|S_f(z)|^2dxdy\in CM(\mathcal{\bb D}.$$
 It remains to show that this last property implies that $\log f'\in BMOA(\bb D)$; this has been proven in  \cite{AZ} in the case $f$ conformal (that is when $G=\{Id\}$). But, thanks to Pommerenke's result, the proof of Lemma 2 in \cite{AZ} goes true for universal coverings, the set
 $$\{g:\,\bb D\to \bb C\,\mathrm{holomorphic\, and\, locally \,injective}: \Vert \log g'\Vert_\mathcal{B}\leq M\}$$being for every $M>0$ a normal family.

\section{Denjoy Domains}
We are now in position to prove our main theorem, whose statement we recall:
\begin{theo} Let $\Omega$ be a Carleson-homogeneous Denjoy domain and $f$ one of its universal coverings, then $\log f'\in BMOA(\bb D).$
\end{theo}
Proof: Let us thus consider a Carleson-homogeneous domain $\Omega$ and let us consider a Fuchsian group $G$ as constructed by Rubel and Ryff that uniformizes $\Omega$.

\begin{lem} Let $\mathcal{F}$ be the fundamental domain for $G$ constructed by Rubel and Ryff. Then $\mathcal{F}$ is a chord-arc domain.
\end{lem}

Notice that this result is close to Gonzalez's one. except that we specify the fundamental domain.  Notice also that the lemma actually holds for more general uniformly perfect Denjoy domains.

\begin{proof} We start by recalling Rubel and Ryff's construction, see \cite{RR}. Let $F$ be a compact subset of the unit circle supposed to be symmetric with respect to $\R$: its complement in the circle is a countable union of intervals $I_j$. Our fundamental domain $\mathcal{F}$ will be the hyperbolic convex hull of $F$ which consists in the domain obtained by replacing each $I_j$ by the hyperbolic geodesic $L_i$ with the same endpoints. Now let $f$ be a Riemann map from $\mathcal{F}\cap \{y>0\}$ onto the half-plane $\{y>0\}$ fixing $-1$ and $1$. We extend $f$ by Schwarz reflection to $\mathcal{F}$ as an isomorphism from $\mathcal{F}$ onto $\mathbb C\backslash\R\cup (-1,1)$. We finally extend $f$ to all of the unit disk by successive Schwarz reflections using the group $G$ generated by the $S\circ R_j$ with $S(z)=\bar z$ and $R_j$ being the reflection across $L_j$. This extended function is then the universal covering of some Denjoy domain and Rubel and Rieff have shown that for every $K\subset \bar{\R}$ a universal covering of $\Omega=\mathbb{C}$ maybe obtained by this method with some $F$ as above.

For the rest of the proof of the lemma we will only use the fact that $\Omega$ is uniformly perfect: Pommerenke \cite{P1} has proven that for such domains there exists $c>0$ such that for every $\gamma$ belonging to a uniformizing Fuchsian group, we have $$\text{Trace}(\gamma)\geq 2+c.$$

Now the boundary of $\mathcal{F}$ is the graph of a function in polar coordinates. So we must show that for any interval $I$ of the unit circle, the length of the part of the graph lying over $I$ is controlled by the distance of the two endpoints of this part of the graph.

Suppose first that the two end points of $I$ lie in $F$: then obviously the length of the graph above $I$ is less that $\pi |I|$ and $I$ is also in this case the distance between the endpoints. Suppose now that one of the endpoints of $I$ is in $E$ while the other is in some $I_j$. Let $C_j$ the part of the graph lying over $I_j\cap I$: if $$\text{length}\, (C_j)\leq 10 |I|,$$ we are essentially back to the proceeding case. If  $$\text{length}\, (C_j)\geq 10 |I|,$$ then the total length over $I$ is controlled by length$(C_j)$ and we are also done.

Remains the most interesting case, i.e. when both endpoints of $I$ lie in say $I_j,I_k$ with $k\neq j$ (the case $k=j$ is obvious). For simplicity of the computations we work in the upper-half plane instead of the disk. We denote by $r_j,r_k$ the length of the intervals $I_j,I_k$ and by $\varepsilon$ the distance between the two intervals. A lengthy but elementary computation shows that, if $R_j,R_k$ denote the reflections through $L_j,L_k$ then
$$\mathrm{Trace}\,(R_jR_k)=2+\frac{2\varepsilon}{\frac{1}{r_j}+\frac{1}{r_k}}.$$
Suppose without loss of generality that $r_j\leq r_k$. Then by Pommerenke's result
$$ |I|\geq \varepsilon\geq \frac{cr_j}{2}$$
and we are back to a preceeding case. This complete the proof of the lemma.
\end{proof}

With this lemma, the proof of Theorem 5.1 is now an immediate corollary of Theorem 4.2 and of Fernandez's $($\cite{F}$)$ result. Indeed, the image $U$ of $\mathcal{F}$ by $f$ is of the form $\bb C\backslash \R\cup (a,b)$ for some $a<b\in \R$ which is obviously an Ahlfors-David regular domain and Fernandez proved that $G$ has $(SFLT)\Leftrightarrow (H)$ by proposition 4.1.
\end{proof}

We believe that the reciprocal of the last theorem is true, i.e. if $\log{f'}\in BMOA(\mathbb{D})$ for all universal coverings of a Denjoy domain then this domain is Carleson-homogeneous: one possible way to prove that on each circular arc the map $f$ "behaves" like the corresponding Joukovsky map $f_{0}$, thus allowing to get a lower bound of the form
$$\iint_{\mathcal{F}}(1-|z|^{2})^{3}|\gamma'(z)|dxdy\geq c\text{diam}(\gamma(\partial \mathcal{F}))$$
uniformly on $\gamma\in G.$ But despite many efforts we could not make this concrete.

\end{document}